\documentclass[12pt]{amsart}

\newcommand{\V}{{\mathfrak V}}

\newcommand{\Ni}{{\mathfrak N}}
\newcommand{\B}{{\mathfrak B}}

\newcommand{\X}{{\mathfrak X}}
\newcommand{\Y}{{\mathfrak Y}}

\newcommand{\U}{{\mathfrak U}}
\newcommand{\A}{{\mathfrak A}}

\newcommand{\Aut}[1]{\mathrm{Aut}\left( #1 \right)}

\newcommand{\var}[1]{\mathrm{var}\left( #1 \right)}

\newcommand{\Wr}{\,\mathrm{Wr}\,}
\newcommand{\Wrr}{\,\mathrm{wr}\,}

\theoremstyle{definition}
  
\theoremstyle{plain}
  
\newtheorem{Theorem}{\sc Theorem}  
\newtheorem{Proposition}{\sc Proposition}

\theoremstyle{remark}
 
\newtheorem{Example}{\sc Example} 

\begin{document}



\title[Varieties generated by wreath products]{Varieties generated by wreath products of abelian and nilpotent groups}
\author{Vahagn H.~Mikaelian}
\thanks{Delivered as a plenary talk at the International Conference ``Mal'tsev Meeting'', Novosibirsk, Russia, November 10-13, 2014. 
The concise versions of this text will appear in Russian in {\it ``Algebra i Logika''} (L.S.~Sobolev Mathematics Institute) and in English in {\it ``Algebra and Logic''} (Springer).
The author was supported in part by joint grant 13RF-030 of RFBR and SCS MES RA, and by 13-1A246 grant of SCS MES RA}
\address{Informatics and Applied Mathematics Department, Yerevan State University, Yerevan 0025, Armenia.}
\email{v.mikaelian@gmail.com}

\date{12 November 2014}


\begin{abstract}
We present a survey of our recent research on varieties, generated by wreath products of groups. In particular, the full classification of all cases, when the (cartesian or direct) wreath product of any abelian groups $A$ and $B$ generates the product variety $\var{A} \var{B}$, is given. The analog of this is given for sets of abelian groups.
We also present new, unpublished research on cases when the similar question is considered for non-abelian, finite groups. 
\end{abstract}

\maketitle


\section{Introduction and background information}
\label{Introduction}

\noindent
The aim of this note is to present our research of recent years on varieties, generated by (standard) wreath products of abelian groups~\cite{AwrB_paper}--\cite{wreath products Prilozh}, and to display some newest, yet unpublished progress in this direction, concerning wreath products of non-abelian groups also. In particular, we give complete classification of all cases when for the abelian groups $A$ and $B$ their cartesian (or direct) wreath product generates the variety $\var{A} \var{B}$. The analog of this also is established for arbitrary {\it sets} of abelian groups. 
%
Since cartesian wreath product $A \Wr B$ and the direct wreath product $A \Wrr B$ of any groups $A$ and $B$ always generate the same variety of groups~\cite[Statement 22.31]{HannaNeumann}, below we will formulate all results for cartesian wreath products, keeping in mind that their analogs also are true for direct wreath products.

\vskip5mm
Wreath products perhaps are the most useful tool to study the product varieties of groups. 
A product $\U \V$ of varieties $\U$ and $\V$ is defined as the variety of extensions of all groups $A \in \U$ by all groups $B \in \V$. By the Kaloujnine-Krassner theorem extensions of $A$ by $B$ can be embedded into the cartesian wreath product $A \Wr B$~\cite{KaloujnineKrasner}. 
Take $A$ and $B$ to be some fixed groups, generating the varieties $\U$ and $\V$ respectively. If 
\begin{equation}
\label{EQUATION_general}    
\var{A \Wr B}=\U \V,
\end{equation}
then the extensions of $A$ by $B$ already are enough to generate $\U \V =\var{A} \var{B}$, and we can restrict ourselves to consideration of $\var{A \Wr B}$, which is easier to study rather than to check all the extensions in $\U \V$.

Examples, when this approach is used, are too many to list, so let us recall the earliest ones (see Chapter 2 of Hanna Neumann's monograph~\cite{HannaNeumann} for references). Using the relatively free groups $F_\infty(\U)$ of infinite rank for varieties $\U$, we get an easy consequence of the Kaloujnine-Krassner theorem:
\begin{equation}
\label{EQUATION_free_groups}
 \var{F_\infty(\U) \Wr {F_\infty(\V)}}= \var{F_\infty(\U)} \var{F_\infty(\V)} =\U \V.
\end{equation}
Moreover, according to~\cite[Statement 22.23]{HannaNeumann}, in (\ref{EQUATION_free_groups}) the group $F_\infty(\U)$ can be replaced by any group $A$, generating $\U$. Namely: 
$$ \var{A \Wr {F_\infty(\V)}}= \var{A} \var{F_\infty(\V)} =\U \V.\
$$ 
Further, by the theorem of G.~Baumslag and Neumanns~\cite{B3,B+3N}, the group $F_\infty(\V)$ can be replaced by any group that discriminates $\V$ because for any group $A$ generating $\U$ and for any group $B$ discriminating $\V$    the wreath product $A \Wr B$ discriminates and generates $\U \V$. We will below get the converse of this for abelian groups in Theorem~\ref{Baumslag's_converse_groups}. 
Let $B^{\infty}$ be the  countably infinite direct power of a group $B$, and let $B^{k}= B \times \cdots \times B$ be the $k$'th direct power of $B$. 
Since $B^{\infty}$ discriminates $\var{B}$, we for any generator $B$ of $\V$ have: 
$$ \var{A \Wr B^{\infty}}= \var{A} \var{B^{\infty}}= \var{A} \var{B} =\U \V.
$$ Another example is when $\U \V$ has a finite basis rank $n$ ($\U \V$ is generated by its $n$-generator groups) and  $\V$ is locally finite. Then 
$$\var{F_t(\U) \Wr {F_n(\V)}}= \U \V,$$ 
where $t = (n-1)|F_n(\V)|+1$.

\vskip5mm
In this note we are mainly interested in the case of wreath products of abelian groups. The initial result in this direction is proved by G.~Higman (Lemma 4.5 and Example 4.9 in~\cite{Some_remarks_on_varieties}): if $C_p$ and $C_n$ are finite cycles of orders respectively $p$ and $n$ ($p$ is a prime not dividing $n$), then 
$$\var{C_p \Wr C_n}= \var{C_p} \var{C_n}= \A_p \A_n,$$ where, as usual, $\A_n$ is the variety of all abelian groups of exponent dividing $n$. Higman used them to study the varieties, which cannot be generated by their $k$-generator groups for any fixed $k$. 
C.H.~Houghton generalized this for the case of arbitrary finite cycles $A=C_m$ and $B=C_n$. Namely, 
$$\var{C_m \Wr C_n}= \var{C_m}  \var{C_n} = \A_m \A_n$$ 
holds if and only if $m$ and $n$ are coprime 
(this result of Houghton is unpublished, but is repeatedly mentioned  in literature, in particular, by R.G.~Burns in~\cite[page 356]{Burns65} and~\cite{BurnsDiso}, and by Hanna Neumann in~\cite{HannaNeumann}.
Some other cases can be obtained using the discriminating groups. For example, the infinite cyclic group $C$ discriminates the variety of all groups $\A$, so for any abelian group $A$ the equality 
$$\var{A \Wr C}= \var{A}  \var{C} = \var{A} \A$$ holds. Also, the infinite direct (or cartesian) power $C_p^\infty$ discriminates the variety $\A_p$, so we have: 
$$\var{A \Wr C_p^\infty}= \var{A}  \var{C_p^\infty} = \var{A} \A_p.$$
On the other hand $\var{C_p \Wr C_p^k}\not= \A_p \A_p$ 
because $C_p \Wr C_p^k$ is nilpotent and, thus, it does not generate the non-nilpotent variety $\A_p \A_p$ for any positive integer $k$~\cite{HannaNeumann}.

\section{Classification for wreath products of arbitrary abelian groups}
\label{Classification}

In~\cite{AwrB_paper, finitely generated abelian}, generalizing the above mentioned results on abelian groups, we gave a general criterion for varieties, generated by wreath product of arbitrary abelian groups $A$ and $B$. 
Let $B_p=\{b \in B \, | \,  |b| = p^i \text{ for some }i \}$ be the $p$-primary component of $B$. By Pr\"ufer's theorem, every abelian group of finite exponent, in particular $B_p$, is a direct product of finite cycles of prime-power exponents: $B_p = \sum_{i\in I} C_{p^{k_i}}$. Denote by $k'$ the greatest of the exponents $k_i$. Then:

\begin{Theorem}[Theorem 6.1 in~\cite{AwrB_paper}]
\label{classification for groups}
For arbitrary abelian groups $A$ and $B$ the equality
$\var{A \Wr B} =  \var A \var B$
holds if and only if:

\noindent 
{\bf (1)} either at least one of the groups $A$ and $B$ is not of finite  exponent; or

\noindent 
{\bf (2)} if  $\exp A = m$  and $\exp B = n$ are both finite and for every  common prime divisor $p$ of $m$ and $n$, the $p$-primary component 
$B_p = \sum_{i\in I} C_{p^{k_i}}$
of $B$ contains infinitely many direct summands $C_{p^{k'}}$, where $p^{k'}$ is the highest power of $p$ dividing $n$.\end{Theorem}

The analog of this criterion also holds for direct wreath products.
The theorem shows that for abelian groups $\var{A \Wr B}$ may be distinct from $\var A \var B$ only if both $A$ and $B$ are of finite exponents $m$ and $n$ respectively, and there is a prime common divisor $p$ of $m$ and $n$ such that the $B$ can be presented in the form
$
	B = ( C_{p^{k'}} \oplus  \cdots  \oplus C_{p^{k'}} ) \oplus B',
$
where $p^{k'} \slash \hskip-1.5mm | \exp B'$. For example, for distinct primes $p$ and $q$:
$$
	\var{C_p \Wr 
[ C_q 
\oplus 
(C_p  \oplus C_p)  
\oplus 
(C_{p^2} \oplus  \cdots \oplus C_{p^2} \oplus  \cdots ) ]
} = \A_p \A_{q p^2},
$$
$$
	\var{C_p \Wr 
[C_q 
\oplus 
(C_{p^2} \oplus C_{p^2} \oplus C_{p^2})
\oplus 
(C_{p} \oplus  \cdots 	\oplus C_{p} \oplus  \cdots ) ]
} \not= \A_p \A_{q p^2}.
$$

When the abelian groups $A$ and $B$ are finite, Theorem~\ref{classification for groups} simply implies:

\begin{Theorem}
\label{Generalisation of Houghton's result for finite abelian groups}
For arbitrary  non-trivial finite abelian groups $A$ and $B$ of exponents $m$ and $n$ respectively
$
	\var{A\Wrr B}= \var A  \var B = \A_m \A_n
$
holds if and only if $m$ and $n$ are coprime.
\end{Theorem}

Thus, the above mentioned result of Houghton is true not only for cyclic group, but for any finite abelian groups. This fact, however, seems to be known in mathematical folklore.
After our paper~\cite{AwrB_paper} was published in 2001, we presented it to Prof.~C.H.~Houghton and asked about the article where his result on wreath product of cyclic groups is published. He confirmed that it was never published: when his advisor Hanna Neumann was preparing the ``Varieties of Groups''~\cite{HannaNeumann}, C.H.~Houghton informed her about this result, and it was mentioned in the book. Later, however, some problem in proofs was discovered, and the article was never published.


\vskip5mm

For any group sets $\X$ and $\Y$ their cartesian wreath product is defined by $\X \Wr \Y = \{ X\Wr Y \,|\, X\in \X, Y\in \Y\}$ (the direct wreath product $\X \Wrr \Y $ is defined analogously). In~\cite{Metabelien, wreath products Prilozh} we discuss the varieties, generated by wreath products of sets of abelian groups. The following theorem gives full classification for this case also:

\begin{Theorem}[Theorem 7.1 in~\cite{Metabelien}]
\label{classification for groups}
For arbitrary sets of abelian groups $\X$ and $\Y$ the equality
$\var{\X \Wr \Y} = \var{\X} \var{\Y}$
holds if and only if:

\noindent 
{\bf (1)} either at least one of the sets $\X$ and $\Y$ is not of finite  exponent; or

\noindent 
{\bf (2)} if  $\exp \X = m$  and $\exp \Y = n$ are both finite and  for every  common prime divisor $p$ of $m$ and $n$, and for arbitrary positive integer $s$ the set $\Y$ contains a group $B(s)$ such that 
the $p$-primary component 
$B(s)_p = \sum_{i\in I(s)} C_{p^{k_i}}$
of $B(s)$ contains at least $s$ direct summands $C_{p^{k'}}$, where $p^{k'}$ is the highest power of $p$ dividing $n$.
\end{Theorem}

Again, the analog of this criterion also holds for direct wreath products. The classification we obtained can also be used to get other results, which are listed in~\cite{AwrB_paper}--\cite{wreath products Prilozh}. For example, the above mentioned theorem of Baumslag and Neumanns~\cite{B3,B+3N} about wreath products with discriminating groups in fact is a {\it necessary and sufficient} condition for abelian groups:

\begin{Theorem}
\label{Baumslag's_converse_groups}
For arbitrary non-trivial abelian groups $A$ and $B$ the wreath product $A \Wr B$ discriminates the variety $\var{A} \var{B}$ if and only if $B$ discriminates the variety $\var{B}$.
\end{Theorem}

The analog of this also holds for wreath products of sets of abelian groups.

\section{Wreath products of finite groups in other classes}
\label{groups_close}

After we fully classified all the cases, when the equality (\ref{EQUATION_general}) holds for abelian groups, it is natural to consider the same problem for non-abelian groups. 
Clearly, this is a much more complicated objective, as there are continuum varieties of groups, and there are non-abelian varieties, in which all finite groups are abelian (see, for example, the result of A.Yu.~Olshanskii~\cite{{Olshanskii all finite groups are abelian}}), whereas the set of abelian and metabelian varieties is countable (see the theorem of D.E.~Cohen~\cite{CohenMetabelian}). We have posed these as problems in~\cite{AwrB_paper, Metabelien, Two problems}.

\vskip3mm

The simplest step to start with is consideration of the varieties $\var{A \Wr B}$ for not necessarily abelian, {\it finite} groups $A$ and $B$. For this case there is a necessity condition based on a theorem of A.L.~Shmel'kin~\cite{ShmelkinOnCrossVarieties}, which strongly restricts the groups we should deal with:

\begin{Proposition}[a corollary of Shmel'kin's theorem]
\label{Shmelkin's Corollary}
If the non-trivial varieties of groups $\U$ and $\V$ are generated by finite groups $A$ and $B$ respectively, then 
$
\var{A \Wr B}=\var{A} \var{B}=\U \V
$
holds only if $B$ is abelian, $A$ is nilpotent, and the exponents of $A$ and $B$ are coprime.
\end{Proposition}
\begin{proof}
By Shmel'kin's theorem $\U \V$ can be generated by a finite group if and only if $\V$ is abelian, $\U$ is nilpotent, and the exponents of $\V$ and $\U$ are coprime~\cite{ShmelkinOnCrossVarieties}. It remains to recall that the wreath product $A \Wr B$ is finite of order $|A|^{|B|}\cdot |B|$.
\end{proof}

Denote by $\Ni_{c,m}$ the variety of all groups of nilpotency class at most $c$ and of exponent dividing $m$. Clearly $\Ni_{c,m} = \Ni_c \cap \B_m$ holds. In this notation Proposition~\ref{Shmelkin's Corollary} can be formulated shorter: {\it If (\ref{EQUATION_general}) holds for finite non-trivial groups $A$ and $B$, then  $A \in \Ni_{c,m}$ and $B \in \A_n$, where $m$ and $n$ are coprime}. The condition of Proposition~\ref{Shmelkin's Corollary} is not sufficient. An example showing this for some small values for nilpotency class and for exponents (thus, for a case ``very near" to the abelian groups) is displayed in~\cite{Metabelien,  wreath products Prilozh}:


\begin{Example}
\label{smallExample}
Let $A=F_2(\Ni_{2,3})$ be the free group of rank $2$ in the variety $\Ni_{2,3}$. Then exponents of $A$ and $C_2$ are coprime, but, nevertheless,
$
	  \var{A \Wr C_2}\not= \var{A} \var{C_2} =\Ni_{2,3}  \A_2.
$
\end{Example}
\begin{proof}
Let us present $A$ as
$
        A=F_2(\Ni_{2,3}) =\langle x_1,x_2 \,|\,  [x_1,x_2,x_1]=[x_1,x_2,x_2]=x_1^3=x_2^3=1 
\rangle
$
and define a group $R$ to be the extension of $A$ by means of the group of 
operators generated by automorphisms $\nu_1,\nu_2\in \Aut{A}$ given by:
$\nu_1 :  x_1\mapsto x_1^{-1}, 	\hskip0mm
\nu_1 :	 x_2\mapsto x_2; 	\hskip1mm
\nu_2 :  x_1\mapsto x_1, 	\hskip0mm 
\nu_2 :  x_2\mapsto x_2^{-1}
$.
Clearly $\langle \nu_1,\nu_2 \rangle \cong C_2 \oplus C_2 \in \A_2$. As it is shown in~\cite{BurnsDiso}, $R$ is a critical group, that is, a finite group which does not belong to the variety generated by its proper factors. Each of its proper factors, but not $R$ itself, satisfies the identity
$
        [[x_1,x_2],[x_3,x_4],x_5]\equiv 1.
$
On the other hand the wreath product $A \Wr C_2$ satisfies this identity because its second commutator subgroup lies in the center. Thus, $R \notin \var{A \Wr C_2}$.
\end{proof}

The following is a partial converse of Proposition~\ref{Shmelkin's Corollary} in the sense that even if (\ref{EQUATION_general}) does not hold for groups $A$ and $B$, satisfying Shmel'kin's conditions, it will hold if we replace $B$ by its finite direct power:

\begin{Proposition}
\label{Nilpotent by abelian}
Let $A \in \Ni_{c,m}$ be any finite group of nilpotency class at most $c>1$ and of exponent $m$, and let $B$ be any finite abelian group of exponent $n$, where $m$ and $n$ are coprime. Then 
$
	  \var{A \Wr B^{c+1}} = \var{A} \var{B} = \var{A} \A_n
$
holds. And, moreover, if $A$ generates $\Ni_{c,m}$, then:
$
	  \var{A \Wr B^{c}} = \var{A} \var{B} = \Ni_{c,m} \A_n.
$
\end{Proposition}
We omit the case $c=1$ because in that case $\Ni_{c,m}$ is abelian, and we covered that case earlier. Also, notice that in the second equality we took lower direct power of $B$. 
\begin{proof}
Let us start from the second equality. As it is proved in~\cite{Burns65, BurnsDiso}, in current circumstances  $\Ni_{c,m} \A_n$ has basis rank $c$, that is, this variety is generated by its $c$-generator groups. 
By~\cite[Proposition 21.13]{HannaNeumann} the $c$-generator free group $F_c(\Ni_{c,m} \A_n)$ is an extension of the $t$-generator free group $F_t(\Ni_{c,m})$ of the variety $\Ni_{c,m}$ by the $c$-generator free group $F_c(\A_n)=C_n^c$ of the variety $\A_n$, where the rank $t$ is calculated by Schreier's formula: $t = (n-1)|F_c(\A_n)|+1 = (n-1)n^c + 1$. 
By the Kaloujnine-Krassner theorem~\cite{KaloujnineKrasner} this extension is embeddable into the cartesian wreath product $F_t(\Ni_{c,m}) \Wr C_n^c$. Since the latter wreath product still is within the product variety $\Ni_{c,m} \A_n$, we get that $\var{F_t(\Ni_{c,m}) \Wr C_n^c} = \Ni_{c,m} \A_n$.

Since $A$ generates $\Ni_{c,m}$,  the group $F_t(\Ni_{c,m})$ by~\cite[Theorem 15.4]{HannaNeumann} can be embedded into the cartesian (in this case also direct, as $A$ and $t$ are finite) power $A^{A^t}$. Thus: $\var{A^{A^t} \Wr C_n^c} = \Ni_{c,m} \A_n$.
Applying~\cite[Lemma 1.1]{AwrB_paper} for the case when $\X = \{A\}$, \, $\Y = \{C_n^c\}$ and $X^*=A^{A^t}$, we get that
$A^{A^t} \Wr C_n^c \in \var{A \Wr C_n^c}$. There only remains to notice that if the exponent of $B$ is $n$, then the direct power $B^c$ contains a subgroup isomorphic to $C_n^c$. Combining these steps, we get:
$$
F_c(\Ni_{c,m} \A_n) \hookrightarrow
F_t(\Ni_{c,m}) \Wr C_n^c \hookrightarrow A^{A^t} \Wr C_n^c 
\in \var{A \Wr C_n^c} \subseteq \var{A \Wr B^c}.
$$

\vskip4mm
Turning to the proof of the first equality, notice that $\var{A}$ may be a proper subvariety of $\Ni_{c,m}$. This also includes the case, when $A$ is abelian, so an even stronger equality holds: $\var{A \Wr B} = \var{A} \var{B} = \A_m \A_n$.

Being a locally finite variety, $\Ni_{c,m} \A_n$ is generated by its critical groups~\cite[Proposition 51.41]{HannaNeumann}. According to~\cite[Statement 2.1]{Burns65} all critical groups in $\Ni_{c,m} \A_n$ are at most $(c+1)$-generator. Each of such critical groups $K$ is an extension of a group $P \in \Ni_{c,m}$ by an at most $(c+1)$-generator subgroup $Q \in \A_n$. Assume the group $P$ is $l$-generator for some $l$ (it is finitely generated, as it is a subgroup of finite index in a finitely generated group).

 $K$ is embeddable into the wreath product $P \Wr Q$. The group $P$ is an epimorphic image of the free group $F_l(\var{A})$ of rank $l$. The latter is embeddable into the cartesian power $A^{A^l}$. So we have two differences from the previous case: $F_l(\var{A})$ may not be equal to $F_l(\Ni_{c,m})$ and $P$ may not be embeddable into $A^{A^l}$. But $P$ still is a factor of $A^{A^l}$, that is, an epimorphic image of a subgroup $S$ of $A^{A^l}$. 

Applying~\cite[Proposition 22.11]{HannaNeumann} (or~\cite[Lemma 1.1]{AwrB_paper}), we get that 
$P \Wr Q$ belongs to $\var{S \Wr Q}$. By~\cite[Proposition 22.12]{HannaNeumann} (or by~\cite[Lemma 1.1]{AwrB_paper}) we also have $S \Wr Q \in \var{A^{A^l} \Wr Q}$. We can again apply~\cite[Lemma 1.1]{AwrB_paper} to get that $A^{A^l} \Wr Q \in \var{A \Wr Q}$. It remains to notice that the exponent $n'$ of $Q$ is a divisor of $n$ and, thus, $Q$ is a subgroup of $C_{n'}^{c+1}$. The latter is a subgroup of $B^{c+1}$ because $B$ has exponent $n$ and contains at least one direct summand $C_n$ (clearly $C_{n'} \hookrightarrow C_n$). By~\cite[Proposition 22.13]{HannaNeumann} (or by~\cite[Lemma 1.2]{AwrB_paper}) we get that $A \Wr Q \in \var{A \Wr B^{c+1}}$. Combination of these steps gives:
$$
K \hookrightarrow P \Wr Q 
\in \var{S \Wr Q} 
\subseteq \var{A^{A^t} \Wr Q} 
\subseteq \var{A \Wr Q} 
\subseteq \var{A \Wr B^{c+1}},
$$
which completes the proof.
\end{proof}

In contrast with the previous example, we get:
\begin{Example}
\label{smallExampleContinued}
If $A=F_2(\Ni_{2,3})$ is the group of Example~\ref{smallExample}, then:
$
	  \var{A \Wr (C_2 \oplus C_2)} = \var{A} \var{C_2 \oplus C_2} =\Ni_{2,3}  \A_2.
$
\end{Example}
\begin{proof}
We just apply the second equality of Proposition~\ref{Nilpotent by abelian} and take into account that the bases rank of the variety $\Ni_{2,3} \A_2$ is 2 (see~\cite{Burns65, BurnsDiso}).
\end{proof}

Comparison of Example~\ref{smallExample}, Proposition~\ref{Nilpotent by abelian} and  Example~\ref{smallExampleContinued} displays how different situation we may have if we proceed from consideration of wreath products of abelian groups to wreath products of nilpotent and abelian groups. 
They also show, that for finite groups, satisfying Shmel'kin's condition, the equality
$
\var{A \Wr B}=\var{A} \var{B}
$
always holds, if the group $B$ is ``large'' in the sense that it contains sufficiently many copies of $C_n$. Thus, the equality can only fail for ``small'' groups $B$, which makes their further study and complete classification a reasonable objective.

\end{document}